\documentclass[11pt,a4paper]{article}
\usepackage[T1]{fontenc}
\usepackage{lmodern}
\usepackage[margin=25mm]{geometry}
\usepackage{amsmath,amssymb,amsthm,mathtools,booktabs,array}
\usepackage{xcolor}
\definecolor{linkblue}{RGB}{0,76,153}
\usepackage{microtype}
\usepackage{enumitem}
\usepackage[hidelinks]{hyperref}
\hypersetup{pdftitle={Unsupported Cyclotomic Divisors in Three-Prime Integer Tilings},
 pdfauthor={Hu Tan and Ying Zhang}}
\allowdisplaybreaks[2]
\newtheorem{theorem}{Theorem}[section]
\newtheorem{lemma}[theorem]{Lemma}
\newtheorem{corollary}[theorem]{Corollary}
\newtheorem{proposition}[theorem]{Proposition}
\theoremstyle{definition}

\newtheorem{remark}[theorem]{Remark}
\newcommand{\ZZ}{\mathbb Z}
\newcommand{\QQ}{\mathbb Q}

\newcommand{\NN}{\mathbb Z_{\geq0}}
\newcommand{\E}{\mathcal E}
\newcommand{\Ap}{\operatorname{Ap}}
\newcommand{\Stab}{\operatorname{Stab}}
\newcommand{\supp}{\operatorname{supp}}
\newcommand{\CT}{\operatorname{CT}}
\newcommand{\ind}{\mathbf 1}
\newcommand{\ii}{\mathrm i}
\newcommand{\br}[1]{\{0,\ldots,#1-1\}}
\title{Unsupported Cyclotomic Divisors\\in Three-Prime Integer Tilings}
\author{\begin{tabular}[t]{c}
Hu Tan\textsuperscript{1}\\[3pt]
\small\textsuperscript{1}Academy of Mathematics and Systems Science\\
\small Chinese Academy of Sciences\\
\small Beijing 100190, China\\
\small\texttt{tanhu2020@amss.ac.cn}
\end{tabular}
\and
\begin{tabular}[t]{c}
Ying Zhang\textsuperscript{2}\\[3pt]
\small\textsuperscript{2}School of Mathematical Sciences\\
\small Soochow University\\
\small Suzhou 215006, China\\
\small\texttt{yzhang@suda.edu.cn}
\end{tabular}}
\date{}
\begin{document}
\maketitle

\begin{abstract}
Cyclotomic divisibility imposes strong prime-power structure on integer tiles. We study unsupported cyclotomic divisors: mixed-order divisors for which none of the prime-power components of the order divides the mask, although every prime in the order divides the tile cardinality. Kiss, \L aba, Marshall and Somlai asked whether such a phenomenon can occur in the three-prime setting.
We prove that unsupported cyclotomic divisors already occur for periods with three distinct prime factors. For primes \(p<q<r\), we characterize the square-period case: an unsupported factor \(\Phi_{pqr}\) occurs in a tiling of \(\ZZ_{(pqr)^2}\) if and only if \(r\in\langle p,q\rangle\), and every such tile lies in a single residue class modulo \(r\). Among cyclic tilings with the unsupported order dividing the specified modulus, the smallest modulus is \(180\); if the order has three distinct prime factors, it is \(900\). An Ap\'ery-set construction gives examples for every triple at period \(p^2q^2r^3\).
The proof of our results combines Fourier rigidity, a three-cylinder decomposition, and an integer mass obstruction. 

\end{abstract}


\section{Introduction and main results}
The cyclotomic zeros of an integer tile describe its distribution across residue classes. The Coven--Meyerowitz conditions link zeros at prime-power orders to zeros at their mixed products. A converse implication would impose additional rigidity: a mixed-order zero, whose prime divisors all divide the tile cardinality, would require a zero at one of its prime-power components. We study the failure of this implication for periods with three distinct prime factors, and classify the structure forced when each prime occurs to exponent two.

For a finite set $A\subset\NN$, write
\[
 A(X)=\sum_{a\in A}X^a,\qquad [n]_X=1+X+\cdots+X^{n-1}.
\]
We use the same notation for the mask of a subset of $\ZZ_M$, always taking its representatives in $\br M$. Let $\Phi_s$ be the $s$th cyclotomic polynomial and put
\[
 S_A=\{p^\alpha:p\text{ prime},\ \alpha\ge1,\ \Phi_{p^\alpha}\mid A(X)\}.
\]
The Coven--Meyerowitz conditions are
\begin{align*}
 \mathrm{(T1)}\quad &|A|=\prod_{s\in S_A}\Phi_s(1),\\
 \mathrm{(T2)}\quad &\Phi_{s_1\cdots s_j}\mid A(X)
 \quad\text{whenever }s_1,\ldots,s_j\in S_A
 \text{ are powers of distinct primes}.
\end{align*}
Coven and Meyerowitz proved that (T1) and (T2) are sufficient for tiling the integers, that (T1) is necessary, and that (T2) is necessary when the tile cardinality has at most two distinct prime factors \cite{CM}.

Following Kiss, \L aba, Marshall and Somlai \cite[Definition 1.5]{KLMS}, a divisor $\Phi_s\mid A$, with $s>1$, is \emph{unsupported} if every prime dividing $s$ divides $|A|$, but
\[
 \Phi_{p^\alpha}\nmid A\qquad\text{for every }p^\alpha\parallel s.
\]
When $A$ is considered in a specified cyclic group $\ZZ_M$, all unsupported orders discussed below are required to divide $M$. This restriction makes the relevant cyclotomic divisibility independent of the choice of representatives. The signature $S_A$ remains the full prime-power signature of the ordinary mask polynomial.
In \cite[Question 1.7]{KLMS}, the authors ask whether a set satisfying (T1) and (T2) can have an unsupported divisor. They give a four-prime construction and leave the three-prime case undecided in Section 10.2. Their Theorem 1.12 excludes such divisors when the period has at most two distinct prime factors.

We answer the three-prime existence question and determine exactly when the full three-prime order can occur at square period. Write
\[
 N=pqr,\qquad M=N^2,
\]
where $p<q<r$ are distinct primes, and define
\begin{equation}\label{eq:E}
 \E(p,q,r)=\{A\subset\ZZ_M:|A|=N,\ \Phi_{d^2}\mid A\ (1<d\mid N),\ \Phi_N\mid A\}.
\end{equation}
The definition is symmetric in $p,q,r$; we retain the same notation when discussing a permutation of the three primes. The three squared prime-power factors exhaust the cardinality at $X=1$. Thus every member has $S_A=\{p^2,q^2,r^2\}$, satisfies (T1), (T2), and has unsupported $\Phi_N$. Section~\ref{sec:fourier} gives a direct tiling complement.

A set $A\subset\ZZ_M$ is \emph{confined modulo $s$}, for a prime $s\mid M$, if it is contained in one residue class modulo $s$. This is different from having a nonzero translation period.

\begin{theorem}[Square-period rigidity and existence]\label{thm:main}
Let $p<q<r$ be distinct primes. Every member of $\E(p,q,r)$ is confined modulo $r$. Moreover,
\begin{equation}\label{eq:maincriterion}
 \E(p,q,r)\ne\varnothing
 \quad\Longleftrightarrow\quad
 r\in\langle p,q\rangle
 :=\{ap+bq:a,b\in\NN\}.
\end{equation}
\end{theorem}

The finite-group formulation does not impose an additional hypothesis on the intended tiling problem. The balanced square-period theorem of \L aba and Londner was proved for odd primes in \cite[Theorem 1.2]{LLodd} and for the even case in \cite[Theorem 1.2]{LLeven}; the combined statement is also recorded in \cite[Theorem 1.3]{LL}. Together with prime-power allocation, it gives the following consequence. The reduction is recorded in Section~\ref{sec:prelim}.

\begin{corollary}[Cyclic tilings]\label{cor:tiling-main}
There exists a tiling $A\oplus B=\ZZ_{(pqr)^2}$ for which $\Phi_{pqr}$ is unsupported for $A$ if and only if $r\in\langle p,q\rangle$. In every such tiling,
\[
 |A|=|B|=pqr,\qquad S_A=\{p^2,q^2,r^2\},\qquad S_B=\{p,q,r\},
\]
both factors satisfy (T1) and (T2), and $A$ is confined modulo $r$.
\end{corollary}

For example, no such tiling exists in $\ZZ_{11025}$, since $7\notin\langle3,5\rangle$. The confinement conclusion has an exact form. Define
\[
 \gcd(M,A-A)=\gcd\bigl(M,\{a-a':a,a'\in A\}\bigr),\qquad
 \Stab(A)=\{t\in\ZZ_M:A+t=A\}.
\]
A set is \emph{primitive} if the first quantity is $1$, and \emph{aperiodic} if the second set is $\{0\}$. We prove in Corollary~\ref{cor:primitivequotient} and Theorem~\ref{thm:aperiodicgeneral} that every $A\in\E(p,q,r)$ satisfies
\[
 \gcd(M,A-A)=r,\qquad |\Stab(A)|\in\{1,p,q,pq\}.
\]
Each of the four stabilizer orders occurs whenever $r\in\langle p,q\rangle$, with an aperiodic complement. After translation and division by $r$, every such $A$ becomes a primitive tile of $\ZZ_{p^2q^2r}$ with unsupported $\Phi_{pq}$.

Existence at a larger period holds for every three-prime set. The construction uses the Ap\'ery mask
\[
 [L]_X\Phi_{pq}(X),\qquad L=r^m\in\langle p,q\rangle.
\]
The precise $0$--$1$ criterion follows from the standard numerical-semigroup identity discussed by Moree \cite{Moree}. An explicit higher-scale factor completes (T1) and (T2) without restoring the absent low prime-power zeros. Taking $m\le2$ gives period $p^2q^2r^{m+1}$. Theorem~\ref{thm:complements} classifies every complement of a more general multiscale version and proves that each satisfies (T2).

The least cyclic period for any unsupported order is $180$; for an unsupported order containing three distinct primes it is $900$. Both bounds are attained with both factors aperiodic. The complete layer description at modulus $900$ yields $1\,834\,800$ distinct subsets with unsupported $\Phi_{30}$, including $874\,800$ aperiodic ones. The enumeration is given separately in Appendix~\ref{app:900}.

Several ingredients have established antecedents. The Ap\'ery identity is classical \cite{Moree}. The five-point core at $(2,3,5)$ realizes the CRT configuration in \cite[Example 8.1]{KLMS}; that core alone is not an unsupported example, since its cardinality is coprime to $6$. The higher-scale lift below changes this cardinality constraint while preserving the mixed zero and completing (T1), (T2). Likewise, the spectrum supplied by (T1), (T2) is due to \L aba \cite[Theorem 1.5(i)]{LabaSpectral}, and standard tiling complements are part of the Coven--Meyerowitz construction; see also \cite[Proposition 3.4]{LLmethods}. Our direct arguments specialize these constructions to the coordinates used below.

Vanishing sums of roots of unity provide another relevant comparison. Lam and Leung \cite{LamLeung} characterize the possible weights of such sums by the additive semigroup generated by the prime divisors of their order. At order and weight $pqr$, the weight criterion supplies no obstruction. The square-signature conditions give the support capacities needed for our rigidity theorem. The same semigroup threshold also has an elementary mass interpretation, recorded in Remark~\ref{rem:masscriterion}.

\subsection{Outline of the proof}
The seven square-order zeros make an $N\times N$ Fourier matrix unitary. Column orthogonality gives one point at each high-digit triple and an exact separation condition on the low digits. The equivalent pairwise difference criterion for integer-box tilings is given in \cite[Lemma 12]{BL}; here we derive it by Fourier orthogonality and express it in high and low digits.

If no low-digit function is constant, the separation condition forces three disjoint cylinder supports. Vanishing at order $pqr$ then gives positive integers $n,m,k$ satisfying $p^2n+q^2m+r^2k=pqr$, together with three support-capacity inequalities. A three-rectangle inequality rules these out for $p\ge5$, and short integer arguments treat $p=2,3$. Hence one low-digit function is constant. The remaining character sum is an additive integer matrix, whose row and column divisibilities yield $r\in\langle p,q\rangle$ and exclude confinement at either smaller prime. The Ap\'ery construction gives the converse. The complement classification and the enumeration are subsequent applications.
\section{Polynomial and tiling preliminaries}\label{sec:prelim}
Throughout, $A\oplus B=\ZZ_M$ means that every residue has exactly one representation as $a+b$. Equivalently,
\begin{equation}\label{eq:tiling}
 A(X)B(X)\equiv[M]_X\pmod{X^M-1}.
\end{equation}
All classifications below refer to this specified cyclic group and its divisors.

We recall the prime-power allocation argument of Coven and Meyerowitz \cite[Lemma 2.1]{CM}; see also \cite[Lemma 9.1]{KLMS}. Its short proof is included to specify the levels relative to the chosen period.

\begin{lemma}[Prime-power allocation in a cyclic tiling]\label{lem:allocation}
If $A\oplus B=\ZZ_M$, then for every prime $p\mid M$, the factors $\Phi_{p^j}$, $1\leq j\leq v_p(M)$, are partitioned between $A$ and $B$, and exactly $v_p(|A|)$ of them divide $A$. No additional $p$-power cyclotomic factor outside these levels divides either mask.
\end{lemma}
\begin{proof}
Every indicated factor divides at least one mask by \eqref{eq:tiling}. Distinct monic cyclotomic factors multiply in $\ZZ[X]$, and each of these has value $p$ at $1$. Hence their total allocation, counting any factor assigned twice, is at most $v_p(|A|)+v_p(|B|)=v_p(M)$. Since all $v_p(M)$ levels must be allocated, equality holds, with no repetition or extra level. A prime not dividing $M$ cannot occur as a prime-power factor of either mask because it would divide its cardinality.
\end{proof}

\begin{proof}[Reduction in Corollary~\ref{cor:tiling-main}]
Suppose $A\oplus B=\ZZ_{N^2}$ and $\Phi_N$ is unsupported for $A$. Then $N\mid|A|$. Each of $\Phi_p,\Phi_q,\Phi_r$ divides $B$, so $N\mid|B|$ as well. Cardinalities give $|A|=|B|=N$, and Lemma~\ref{lem:allocation} assigns precisely the squared levels to $A$ and the first levels to $B$. The odd and even balanced square-period theorems \cite{LLodd,LLeven}, in the combined form \cite[Theorem 1.3]{LL}, give (T2) for both. Hence $A\in\E(p,q,r)$, to which Theorem~\ref{thm:main} applies. Conversely, the square-signature normal form in Theorem~\ref{thm:normalform} supplies a cyclic tiling for every member of $\E(p,q,r)$.
\end{proof}

\subsection{Numerical semigroups and the Ap\'ery mask}
Let $p,q$ be distinct primes, $Q=pq$, and $S=\langle p,q\rangle$. The standard semigroup identity is
\begin{equation}\label{eq:semigroup}
 \frac{\Phi_Q(X)}{1-X}
 =\frac{1-X^{pq}}{(1-X^p)(1-X^q)}
 =\sum_{n\in S}X^n.
\end{equation}
For completeness, every element of $S$ has a unique expression $ap+bq$ with $a\geq0$ and $0\leq b<p$: reduce $b$ modulo $p$, transferring multiples of $pq$ into $ap$. Summing these expressions proves the identity. Its connection with binary cyclotomic polynomials, and the Ap\'ery formulation, are classical; see \cite{Moree}.

\begin{lemma}[Ap\'ery smoothing]\label{thm:apery}
For every integer $L\geq1$, the polynomial
\[
 C_L(X)=[L]_X\Phi_{pq}(X)
\]
has all coefficients in $\{0,1\}$ if and only if $L\in S$. When this holds,
\[
 \supp C_L=\Ap(S;L):=\{s\in S:s-L\notin S\}.
\]
This support has exactly $L$ elements and is a complete residue system modulo $L$. Its degree is $L-1+(p-1)(q-1)<QL$.
\end{lemma}
\begin{proof}
By \eqref{eq:semigroup}, its coefficient at $X^n$ is
\[
 \ind_S(n)-\ind_S(n-L),
\]
where membership at negative integers means zero. If $L\notin S$, the coefficient at $n=L$ is $-1$. If $L\in S$, closure of $S$ under addition prevents negative coefficients. The displayed support is then immediate. Each residue class modulo $L$ has a least member in $S$, and all its later members are obtained by adding $L$. Hence the least member is the unique member of the Ap\'ery set in that class. The degree assertion follows from $\deg\Phi_{pq}=(p-1)(q-1)$ and $(p-1)(q-1)<Q$.
\end{proof}

\begin{remark}[General coprime generators]
For general coprime integers $a,b>1$, the right polynomial is
\[
 P_{\langle a,b\rangle}(X)
 =\frac{(1-X)(1-X^{ab})}{(1-X^a)(1-X^b)},
\]
not necessarily $\Phi_{ab}$. The same smoothing criterion applies to $[L]P_{\langle a,b\rangle}$. The assertion with $\Phi_{ab}$ alone is false for arbitrary coprime integers: $2\in\langle2,9\rangle$, but
\[
 [2]_X\Phi_{18}(X)=1+X-X^3-X^4+X^6+X^7.
\]
\end{remark}

\section{Fourier normalization at square period}\label{sec:fourier}
Let $N=pqr$ and $M=N^2$, with distinct primes $p<q<r$. We first impose the seven square-order conditions in \eqref{eq:E}, without imposing the low-order condition $\Phi_N\mid A$. The resulting normal form is necessary and sufficient. The spectrum used in the matrix argument is the square-signature specialization of \cite[Theorem 1.5(i)]{LabaSpectral}. The conclusion about the standard complement is also a specialization of the standard-complement criterion \cite[Proposition 3.4]{LLmethods}, and the pairwise difference condition is the specialization of \cite[Lemma 12]{BL} to exponents two. We give both directions explicitly to retain the precise low-digit condition (K).

Use the standard CRT identification
\[
 \ZZ_M\simeq\ZZ_{p^2}\times\ZZ_{q^2}\times\ZZ_{r^2}.
\]
In each coordinate, separate the high digit from the low digit, and put $I=\br p\times\br q\times\br r$.

\begin{theorem}[Square-signature normal form]\label{thm:normalform}
A set of cardinality $N$ has all the square-signature zeros in \eqref{eq:E} if and only if it can be written uniquely as
\begin{equation}\label{eq:normalform}
 A=\{(pu+f(v,w),\ qv+g(u,w),\ rw+h(u,v)):(u,v,w)\in I\},
\end{equation}
where the low-digit functions take values in $\br p,\br q,\br r$, respectively, and satisfy the following condition:
\begin{quote}
\textbf{(K)} For every two distinct high-digit triples, at least one coordinate has different high digits and equal low digits.
\end{quote}
Every such set tiles with the standard complement
\begin{equation}\label{eq:standardB}
 B_\square=\br p\times\br q\times\br r.
\end{equation}
In addition, $A\in\E(p,q,r)$ if and only if
\begin{equation}\label{eq:lowzero}
 \sum_{u,v,w}\zeta_p^{f(v,w)}\zeta_q^{g(u,w)}\zeta_r^{h(u,v)}=0,
 \qquad \zeta_s=e^{2\pi\ii/s}.
\end{equation}
\end{theorem}
\begin{proof}
Form the $N\times N$ matrix with columns indexed by $a\in A$ and rows indexed by $\omega\in\br p\times\br q\times\br r$:
\[
 U_{\omega,a}=N^{-1/2}\exp\left(2\pi\ii\left(
 \frac{\omega_1a_1}{p^2}+\frac{\omega_2a_2}{q^2}+\frac{\omega_3a_3}{r^2}\right)\right).
\]
For two distinct rows, each nonzero frequency difference in a coordinate is coprime to that coordinate's prime. Their inner product therefore evaluates $A$ at a character whose order is one of the required products of squared primes. Hence the rows are orthonormal and, since $U$ is square, so are the columns.

For two columns, their inner product factors into three geometric sums. For a prime $s$,
\[
 \sum_{j=0}^{s-1}e^{2\pi\ii j\Delta/s^2}=0
 \quad\Longleftrightarrow\quad s\mid\Delta\text{ and }s^2\nmid\Delta.
\]
Thus any two distinct points have, in some coordinate, a difference divisible by its prime but not by its square. At most one point can have any given triple of high digits, since two such points would have all coordinate differences of absolute value less than the respective primes. As there are $N$ points, every high-digit triple occurs once. In these coordinates the difference condition is exactly (K). Comparing two triples differing only in $u$ forces their first low digits to agree. This makes the first low digit independent of $u$, and similarly for the other coordinates. This gives \eqref{eq:normalform}.

Conversely, (K) makes the columns orthonormal, hence the rows. Taking frequency differences equal to $1$ in any chosen nonempty coordinate subset gives a zero of the corresponding square order. Cyclotomic irreducibility gives the full divisibility statement.

To prove tiling with $B_\square$, a nonzero difference of two points of $A$ has some coordinate divisible exactly once by its prime. A difference of two points of $B_\square$ cannot have that property. Hence the two difference sets intersect only in zero, and their cardinality product is $M$. Finally, the character in \eqref{eq:lowzero} has order $N$, giving the last equivalence.
\end{proof}

\begin{lemma}[Two-dimensional consequence]\label{lem:2D}
For a set parametrized by $(pu+f(v),qv+g(u))$ on a complete $p\times q$ high-digit grid, condition (K) holds if and only if $f$ or $g$ is constant.
\end{lemma}
\begin{proof}
If both are nonconstant, choose $v\ne v'$ with $f(v)\ne f(v')$ and $u\ne u'$ with $g(u)\ne g(u')$; the resulting pair violates (K). If one is constant, a pair differing in its corresponding high coordinate is separated there; a pair differing only in the other high coordinate is separated in that other coordinate.
\end{proof}

\section{The three-cylinder alternative}\label{sec:cylinders}
The normal form separates the high digits from the low digits, but condition (K) still couples the three low functions. The next lemma describes that coupling when no low function is constant. Its proof is combinatorial and remains valid for arbitrary coordinate sizes at least two.

\begin{lemma}[Three-cylinder alternative]\label{lem:cylinders}
Suppose that the functions in \eqref{eq:normalform} satisfy (K), and that none of $f,g,h$ is constant. After subtracting constants from the low values in their respective residue groups, there are nonempty proper subsets
\[
 U_0\subset\br p,\qquad V_0\subset\br q
\]
and nonempty disjoint subsets $W_f,W_g\subset\br r$ such that
\begin{equation}\label{eq:cylinder-supports}
 \begin{split}
 \supp f&\subset V_0\times W_f,\\
 \supp g&\subset U_0\times W_g,\\
 \supp h&\subset U_1\times V_1,
 \qquad U_1=\br p\setminus U_0,\quad V_1=\br q\setminus V_0.
 \end{split}
\end{equation}
In particular, at every high-digit triple, at most one of the three low values is nonzero. Conversely, any functions with the support restrictions \eqref{eq:cylinder-supports} satisfy (K).
\end{lemma}
\begin{proof}
View $h$ as a $p\times q$ matrix. A nonconstant row $h(u,\cdot)$ forces $g(u,w)$ to be independent of $w$, by Lemma~\ref{lem:2D} applied with $u$ fixed. A nonconstant column $h(\cdot,v)$ similarly forces $f(v,w)$ to be independent of $w$.

We first show that $h$ has both a constant row and a constant column. If it has no constant row, then $g(u,w)=g_u$. Since $g$ is not globally constant, Lemma~\ref{lem:2D} on every fixed-$w$ layer gives $f(v,w)=f_w$. A nonconstant column of $h$ would now make $f_w$ constant, contrary to the hypothesis. Therefore all columns of $h$ are constant, so $h(u,v)=h_v$. Each of $f_w,g_u,h_v$ is nonconstant. Choose $u\ne u'$, $v\ne v'$, and $w\ne w'$ on which the corresponding functions have different values. The two triples $(u,v,w)$ and $(u',v',w')$ then have different low values in all three coordinates, contradicting (K). Interchanging the first two coordinates proves that $h$ cannot have no constant column either.

Let $U_0$ be the set of constant rows and $V_0$ the set of constant columns. They are nonempty and proper: if, for example, every row were constant, one constant column would make the whole matrix constant. Their common value is some $\gamma$, because a constant row and a constant column meet. Thus $h=\gamma$ whenever $u\in U_0$ or $v\in V_0$.

For $u\in U_1$, write $g(u,w)=g_u$, and for $v\in V_1$, write $f(v,w)=f_v$. If the values $f_v$ on $V_1$ were not all equal, then every fixed-$w$ layer would have a nonconstant $f$-vector, forcing its $g$-vector to be constant. Its value is fixed at any $u\in U_1$, so $g$ would be globally constant. Hence $f_v=\alpha$ on $V_1$. The same argument gives $g_u=\beta$ on $U_1$.

On each fixed-$w$ layer, at least one of the two low vectors is constant. If the $f$-vector is constant, its value is necessarily $\alpha$; if the $g$-vector is constant, its value is $\beta$. Consequently
\[
 W_f=\{w:\text{some }f(v,w)\ne\alpha\},\qquad
 W_g=\{w:\text{some }g(u,w)\ne\beta\}
\]
are nonempty and disjoint. Subtracting $\alpha,\beta,\gamma$ from the low values gives \eqref{eq:cylinder-supports}. This normalization preserves equality of low values and multiplies \eqref{eq:lowzero} by a nonzero constant; it is used only for these two purposes.

For the converse, label a point by $F$, $G$, or $H$ if its nonzero low coordinate is the first, second, or third, and by $O$ if all are zero. Two points of types $F,G$ have distinct high $w$-digits and equal third low digits; types $F,H$ are separated in the second coordinate; types $G,H$ are separated in the first. For two points of type $F$, different $(v,w)$ are separated in the second or third coordinate; if $(v,w)$ agrees, their first low values agree and their high $u$-digits differ. The other equal-type cases are identical. A point of type $F$ and a point of type $O$ cannot have the same $(v,w)$, and are again separated in the second or third coordinate. The remaining cases follow by permutation, or are immediate for two type-$O$ points. Thus (K) holds.
\end{proof}

The support description reduces the vanishing condition to an integer mass identity and three capacity bounds.

\begin{lemma}[Axis multiplicities]\label{lem:axis-multiplicities}
Under the hypotheses of Lemma~\ref{lem:cylinders}, use the normalized functions and the subsets in \eqref{eq:cylinder-supports}. Suppose moreover that \eqref{eq:lowzero} holds. Set
\[
 a=|U_0|,\quad b=|V_0|,\quad c=|W_f|,\quad d=|W_g|.
\]
There are positive integers $n,m,k$ such that
\begin{equation}\label{eq:mass}
 p^2n+q^2m+r^2k=pqr,
\end{equation}
and
\begin{equation}\label{eq:capacities}
 \begin{gathered}
 1\le a\le p-1,\quad 1\le b\le q-1,\quad c,d\ge1,\quad c+d\le r,\\
 (p-1)n\le bc,\qquad (q-1)m\le ad,\qquad
 (r-1)k\le(p-a)(q-b).
 \end{gathered}
\end{equation}
\end{lemma}
\begin{proof}
For $1\le i<p$, let $n_i$ count the pairs $(v,w)$ with $f(v,w)=i$; define $m_j$ and $k_\ell$ similarly. A nonzero first low value occurs at $p n_i$ points, because the high digit $u$ is free. The corresponding counts on the other axes are $q m_j$ and $r k_\ell$. All remaining points have zero low values. Thus the low-character sum is
\begin{equation}\label{eq:axis-zero}
 C+p\sum_{i=1}^{p-1}n_i\zeta_p^i
  +q\sum_{j=1}^{q-1}m_j\zeta_q^j
  +r\sum_{\ell=1}^{r-1}k_\ell\zeta_r^\ell,
\end{equation}
where
\[
 C=pqr-p\sum_i n_i-q\sum_j m_j-r\sum_\ell k_\ell.
\]
The products $\zeta_p^i\zeta_q^j\zeta_r^\ell$, with $0\le i<p-1$, $0\le j<q-1$, and $0\le\ell<r-1$, form a basis of $\QQ(\zeta_{pqr})$. Indeed, the three roots generate this field, their individual minimal polynomials reduce every exponent to the indicated range, and the number of spanning monomials is $(p-1)(q-1)(r-1)=[\QQ(\zeta_{pqr}):\QQ]$. Eliminate the last power in each sum in \eqref{eq:axis-zero}. The coefficient of each nonconstant single-coordinate basis vector shows that
\[
 n_1=\cdots=n_{p-1}=n,\quad
 m_1=\cdots=m_{q-1}=m,\quad
 k_1=\cdots=k_{r-1}=k.
\]
For a prime equal to $2$, the corresponding assertion has just one entry and is automatic. Each common value is positive, since each of the normalized low functions is nonzero somewhere. The constant coefficient of \eqref{eq:axis-zero} is now $pqr-p^2n-q^2m-r^2k$. This proves \eqref{eq:mass}. Counting available positions in the three rectangles of \eqref{eq:cylinder-supports} gives \eqref{eq:capacities}.
\end{proof}

\section{Excluding the unconfined alternative}\label{sec:exclusion}
We now show that the mass identity and the capacity bounds are incompatible.

\begin{lemma}[Three-rectangle inequality]\label{lem:rectangles}
Let $x,y,z\in[0,1]$, and suppose
\[
 0\le\alpha\le yz,\qquad
 0\le\beta\le x(1-z),\qquad
 0\le\gamma\le(1-x)(1-y).
\]
Then
\begin{equation}\label{eq:rectangle-ineq}
 1-\alpha-\beta-\gamma\ge2\sqrt{\alpha\beta\gamma}.
\end{equation}
\end{lemma}
\begin{proof}
Put $A=yz$, $B=x(1-z)$, and $C=(1-x)(1-y)$. The exact identities
\[
 1-A-B-C=D+E,\qquad ABC=DE,
\]
hold with
\[
 D=x(1-y)z,\qquad E=(1-x)y(1-z).
\]
Therefore $1-A-B-C\ge2\sqrt{ABC}$ by the arithmetic--geometric mean inequality. Decreasing $A,B,C$ to $\alpha,\beta,\gamma$ increases the left-hand side and decreases the right-hand side.
\end{proof}

\begin{lemma}[Arithmetic obstruction]\label{lem:arithmetic}
For distinct primes $p<q<r$, the integer system \eqref{eq:mass}--\eqref{eq:capacities} has no solution.
\end{lemma}
\begin{proof}
First suppose $p\ge5$. Apply Lemma~\ref{lem:rectangles} with
\[
 x=\frac ap,\qquad y=\frac bq,\qquad z=\frac cr,
\]
and
\[
 \alpha=\frac{(p-1)n}{qr},\qquad
 \beta=\frac{(q-1)m}{pr},\qquad
 \gamma=\frac{(r-1)k}{pq}.
\]
The capacity bounds apply because $d\le r-c$. By \eqref{eq:mass},
\[
 1-\alpha-\beta-\gamma=\frac{pn+qm+rk}{pqr}.
\]
Consequently \eqref{eq:rectangle-ineq} gives
\begin{equation}\label{eq:lower-mass}
 \frac{(pn+qm+rk)^2}{nmk}\ge4(p-1)(q-1)(r-1).
\end{equation}
On the other hand, $n,m,k\ge1$ and \eqref{eq:mass} imply
\begin{align}
 \frac{(pn+qm+rk)^2}{nmk}
 &=\frac{p^2n}{mk}+\frac{q^2m}{nk}+\frac{r^2k}{nm}
     +\frac{2pq}{k}+\frac{2pr}{m}+\frac{2qr}{n}\notag\\
 &\le pqr+2(pq+pr+qr).\label{eq:upper-mass}
\end{align}
These inequalities are incompatible. Indeed
\begin{align*}
 &4(p-1)(q-1)(r-1)-pqr-2(pq+pr+qr)\\
 &\hspace{12mm}=3pqr-6(pq+pr+qr)+4(p+q+r)-4>0.
\end{align*}
To verify positivity, write $p=5+s$, $q=7+t$, and $r=11+u$, where $s,t,u\ge0$. The last expression is
\begin{equation}\label{eq:positive-certificate}
 241+127s+73t+37u+27st+15su+9tu+3stu.
\end{equation}
The lower bounds $5,7,11$ hold because the primes are distinct and ordered.

\begin{samepage}
If $p=2$, then $a=1$, so the last capacity inequality gives
\[
 r-1\le(r-1)k\le q-b\le q-1,
\]
contrary to $r>q$.
\end{samepage}

It remains to treat $p=3$. If $a=2$, the last capacity inequality gives $r-1\le q-b<r-1$, which is impossible. Thus $a=1$. If $k\ge2$, the same inequality gives $2(r-1)\le2(q-b)<2(r-1)$, again impossible. Hence $k=1$ and
\[
 r\le2q-2b+1.
\]
If $m\ge2$, the second capacity bound and $c+d\le r$ give
\[
 2(q-1)\le d\le r-c\le2q-2b+1-c,
\]
so $2b+c\le3$. Positivity of $b,c$ forces $b=c=1$, contradicting $2n\le bc$. Thus $m=1$. Equation \eqref{eq:mass} becomes
\[
 9n+q^2+r^2=3qr.
\]
Modulo $3$, the left-hand side is $2$, since $q$ and $r$ are primes different from $3$, whereas the right-hand side is $0$. This final contradiction proves the lemma.
\end{proof}

\begin{theorem}[Confinement]\label{thm:confinement}
Every member of $\E(p,q,r)$ is confined modulo at least one of $p,q,r$. Equivalently, at least one of the low functions in \eqref{eq:normalform} is constant.
\end{theorem}
\begin{proof}
If all three low functions were nonconstant, Lemmas~\ref{lem:cylinders} and \ref{lem:axis-multiplicities} would give a solution of \eqref{eq:mass}--\eqref{eq:capacities}, contradicting Lemma~\ref{lem:arithmetic}.
\end{proof}

\begin{remark}[Why the low-order zero is essential]\label{rem:defect}
Condition (K) alone does not force confinement. For $(p,q,r)=(3,5,7)$, take $U_0=\{0\}$, $V_0=\{0,1\}$, $W_f=\{0,1,2\}$, and $W_g=\{3,4,5,6\}$. Assign $f(0,w)=1$ and $f(1,w)=2$ on $W_f$; assign $g(0,3),\ldots,g(0,6)=1,2,3,4$; and assign the six values $1,\ldots,6$ to $h$ on $\{1,2\}\times\{2,3,4\}$. Put every other low value equal to zero. Lemma~\ref{lem:cylinders} gives (K), but the low-character sum is
\[
 105-3^2\cdot3-5^2-7^2=4,
\]
not zero. Thus the square-signature conditions allow unconfined tiles, and the additional low-order zero supplies the obstruction.
\end{remark}
\section{The semigroup criterion and the complete layer form}\label{sec:criterion}
In the CRT normal form, confinement modulo $r$ means exactly that $h$ is constant. We now recover a numerical-semigroup condition from that support restriction.

\begin{theorem}[Confined-support criterion]\label{thm:confined}
For any three distinct primes $p,q,r$ (without an ordering requirement), there is a member of $\E(p,q,r)$ confined modulo $r$ if and only if
\begin{equation}\label{eq:sharpsemi}
 r\in\langle p,q\rangle.
\end{equation}
\end{theorem}
\begin{proof}
Sufficiency is Theorem~\ref{thm:construction} with $L=r$, $e=1$.
For necessity, translate the third low digit to zero, so $h=0$. By Lemma~\ref{lem:2D}, each fixed-$w$ layer has $f(\cdot,w)$ constant or $g(\cdot,w)$ constant. Define
\[
 W_{ij}=\#\{(u,v,w):f(v,w)=i,\ g(u,w)=j\}.
\]
Every entry belongs to $\langle p,q\rangle$: a layer with constant $g$ contributes multiples of $p$ to each entry, while one with constant $f$ contributes multiples of $q$. Independently, every row sum is a multiple of $p$ and every column sum a multiple of $q$.

The zero \eqref{eq:lowzero} becomes $\sum_{i,j}W_{ij}\zeta_p^i\zeta_q^j=0$. The products $\zeta_p^i\zeta_q^j$, $i<p-1$, $j<q-1$, form a basis of $\QQ(\zeta_{pq})$. Eliminating the last row and column therefore gives
\[
 W_{ij}-W_{i,q-1}-W_{p-1,j}+W_{p-1,q-1}=0.
\]
Hence $W$ is an additive integer matrix. Choose a position $(i_0,j_0)$ of minimum value $t$. A row difference is independent of the column; multiplying it by the number $q$ of columns gives a multiple of $p$, so coprimality forces the difference itself to be a multiple of $p$. The analogous argument forces each column difference to be a multiple of $q$. Thus
\begin{equation}\label{eq:matrixnormal}
 W_{ij}=t+p\alpha_i+q\beta_j,\qquad
 \alpha_i,\beta_j\in\NN,\quad\min\alpha=\min\beta=0.
\end{equation}
Put $a=\sum_i\alpha_i$, $b=\sum_j\beta_j$. Total mass and row/column divisibility yield
\[
 r=t+a+b,\qquad t+b=pv,\qquad t+a=qu
\]
for nonnegative integers $u,v$. Since the minimum entry is itself in the semigroup, write $t=pc+qd$, $c,d\geq0$. Then
\[
 r=p(v-c)+q(u-d).
\]
Both coefficients are nonnegative: $pv=t+b\geq pc$ and $qu=t+a\geq qd$. This proves \eqref{eq:sharpsemi}.
\end{proof}

\begin{remark}[An equivalent mass criterion]\label{rem:masscriterion}
For any integer $L\ge1$ coprime to $pq$, membership $L\in\langle p,q\rangle$ is equivalent to the assertion that
\[
 pq=pU+qV+LW,\qquad U,V,W\in\NN,
\]
has only the solutions $(q,0,0)$ and $(0,p,0)$. Indeed, in a representation $L=ap+bq$, coprimality forces $a,b\ge1$. A solution with $W>0$ would express $pq$ as a combination of $p$ and $q$ with both coefficients positive, which reduction modulo $p$ excludes. Conversely, choose $b\in\{1,\ldots,p-1\}$ with $bq\equiv L\pmod p$. If $L\notin\langle p,q\rangle$, then $a=(bq-L)/p\ge1$, and $pq=pa+q(p-b)+L$ gives a solution with $W=1$. Thus the admissible Ap\'ery scales are exactly those for which the third term cannot occur in this mass equation.
\end{remark}

\begin{proof}[Proof of Theorem~\ref{thm:main}]
By Theorem~\ref{thm:confinement}, every member of $\E(p,q,r)$ is confined modulo some prime dividing $M$. Confinement modulo $p$ would give $p\in\langle q,r\rangle$ by Theorem~\ref{thm:confined}, which is impossible. Confinement modulo $q$ would give $q\in\langle p,r\rangle$; since $q<r$ and $p\nmid q$, this too is impossible. Thus confinement is modulo $r$, and Theorem~\ref{thm:confined} proves \eqref{eq:maincriterion}.
\end{proof}

\begin{corollary}[Complete layer form]\label{cor:layerform}
After translation in the third CRT coordinate, the members of $\E(p,q,r)$ are exactly the sets
\begin{equation}\label{eq:layerform}
 A=\{(pu+f_w(v),\ qv+g_w(u),\ rw):
       u\in\br p,\ v\in\br q,\ w\in\br r\},
\end{equation}
where, for every $w$, either $f_w$ or $g_w$ is constant, and
\begin{equation}\label{eq:layerzero}
 \sum_{w=0}^{r-1}
 \left(\sum_{v=0}^{q-1}\zeta_p^{f_w(v)}\right)
 \left(\sum_{u=0}^{p-1}\zeta_q^{g_w(u)}\right)=0.
\end{equation}
\end{corollary}
\begin{proof}
Necessity follows from Theorem~\ref{thm:main}, Theorem~\ref{thm:normalform}, and Lemma~\ref{lem:2D}. For sufficiency, pairs in one $w$-layer satisfy (K) by Lemma~\ref{lem:2D}; pairs in different layers are separated in the third coordinate because their third low digits are all zero. Theorem~\ref{thm:normalform} and \eqref{eq:layerzero} therefore give all the required zeros.
\end{proof}

\begin{corollary}[Primitive quotient and complement slices]\label{cor:primitivequotient}
Let $p<q<r$ be distinct primes, $N=pqr$, $M=N^2$, and
$A\in\E(p,q,r)$. Then
\[
 \gcd(M,A-A)=r,
 \qquad |\Stab(A)|\mid pq.
\]
After translation, write $A=rC$, with $C\subset\ZZ_{M/r}$. Then
\[
 \gcd(M/r,C-C)=1,\qquad S_C=\{p^2,q^2,r\}.
\]
The set $C$ is a tile of $\ZZ_{M/r}$ and satisfies \emph{(T2)}, and $\Phi_{pq}$ is unsupported
for $C$. Conversely, every set $C\subset\ZZ_{M/r}$ of cardinality
$N$ with this signature, satisfying \emph{(T2)} and
$\Phi_{pq}\mid C$, gives $rC\in\E(p,q,r)$.

Every complement of $rC$ has the unique form
\[
 B=\bigcup_{j=0}^{r-1}(j+rD_j),
 \qquad C\oplus D_j=\ZZ_{M/r}\quad(0\le j<r).
\]
In particular, $|D_j|=pq$, and the complements $D_j$ may be
chosen independently.
\end{corollary}
\begin{proof}
Theorem~\ref{thm:main} gives confinement modulo $r$. Confinement
modulo $p$ or $q$ is excluded by Theorem~\ref{thm:confined} and
the ordering of the primes. Confinement modulo $r^2$ would give
$A(\zeta_{r^2})=\zeta_{r^2}^{t}|A|\ne0$ for some integer $t$,
contrary to $\Phi_{r^2}\mid A$. Thus $\gcd(M,A-A)=r$.

The stabilizer acts freely on $A$, so its order divides $N$.
If that order were divisible by $r$, the order-$r$ subgroup of
$\ZZ_M$ would preserve $A$. In the layer form of
Corollary~\ref{cor:layerform}, this would make all $r$ layers
identical. The low-character contribution of each individual
layer is nonzero: one low vector is constant, and vanishing
of the other vector's sum of prime-order roots would require
$p\mid q$ or $q\mid p$. The total low-character sum could
therefore not vanish. Hence $|\Stab(A)|\mid pq$.

Choose the common residue of $A$ modulo $r$ in $\br r$ and
subtract it. This gives the ordinary mask identity
$A(X)=C(X^r)$ for the translated set. The preceding gcd
identity gives primitivity of $C$. Evaluating this mask at
roots of unity shows that substitution by $X^r$ preserves
the $p$- and $q$-power levels and raises every positive
$r$-power level by one. Consequently
$S_C=\{p^2,q^2,r\}$, and the seven square-signature zeros
of $A$ give \emph{(T2)} for $C$. The order-$pqr$ zero of
$A$ gives the order-$pq$ zero of $C$, which is unsupported
by its signature. The same root-order calculation proves
the converse.

Finally, put
$D_j=\{(b-j)/r:b\in B,\ b\equiv j\pmod r\}$, interpreted
in $\ZZ_{M/r}$. Representations of residues congruent to $j$
in $rC+B$ correspond exactly to representations in $C+D_j$.
Thus $rC\oplus B=\ZZ_M$ holds precisely when every $D_j$
is a complement of $C$, proving the last assertion.
\end{proof}

\subsection{All stabilizer orders}
The stabilizer restriction in Corollary~\ref{cor:primitivequotient} is sharp. All four possibilities can be obtained by modifying the same collection of layers.

\begin{theorem}[All stabilizer orders]\label{thm:aperiodicgeneral}
If $p<q<r$ are primes and $r\in\langle p,q\rangle$, the stabilizer orders realized by members of $\E(p,q,r)$ are exactly $1,p,q,pq$. Each occurs in a tiling with the same aperiodic complement $B_\square$.
\end{theorem}
\begin{proof}
Write $r=ap+bq$, where $a,b\geq1$; neither coefficient can vanish because $r$ is a distinct prime. Start with $r$ layers with $h=0$. Choose their constant low pairs $(f,g)$ as $a$ copies of the $p$-cycle $(i,0)$, $i\in\br p$, and $b$ copies of the $q$-cycle $(0,j)$, $j\in\br q$. Each layer is a full high-digit rectangle. The low-character sum is $pq$ times a sum of full $p$- and $q$-cycles, hence zero.

Every low vector is constant in this initial construction, so it has both the order-$p$ and order-$q$ periods. Corollary~\ref{cor:primitivequotient} gives stabilizer order exactly $pq$.

For a first modification, choose the layers $(0,0),(1,0)$ from one $p$-cycle. Replace their constant $f$-vectors by
\[
 (0,\ldots,0,1),\qquad(1,\ldots,1,0),
\]
respectively, leaving $g=0$. Their combined low-character contribution is unchanged. For a second, independent modification, choose two layers $(0,0),(0,1)$ from one $q$-cycle, distinct from the two previously selected layers; leave $f=0$ and replace their constant $g$-vectors by the analogous complementary vectors of length $p$. Again the combined contribution is unchanged. Every layer satisfies Lemma~\ref{lem:2D}, and distinct layers satisfy (K) because $h=0$.

The order-$p$ subgroup cycles the high $u$-digits, so it preserves $A$ exactly when every $g_w$ is constant. Likewise, the order-$q$ period is present exactly when every $f_w$ is constant. Applying only the first modification therefore gives stabilizer order $p$; applying only the second gives order $q$; applying both gives order $1$. No order-$r$ period is possible by Corollary~\ref{cor:primitivequotient}. All four sets tile with $B_\square$ by Theorem~\ref{thm:normalform}. Its coordinate intervals are aperiodic, so their product is aperiodic.
\end{proof}

In particular, the semigroup criterion also characterizes existence with both factors aperiodic. Every $A\in\E(p,q,r)$ has the precise common divisor $r$, while its quotient by $r$ is primitive.
\section{Ap\'ery constructions and all complements}\label{sec:apery-tilings}
Fix distinct primes $p,q,r$, integers $m\geq1$, $e\geq0$, and set
\[
 Q=pq,\quad L=r^m\in\langle p,q\rangle,\quad h=r^e,\quad
 K=QhL,\quad M=QK.
\]
The parameter $e$ controls the number of unsupported scales, while $m$ controls the Ap\'ery core size.

\begin{theorem}[Construction with arbitrarily many unsupported scales]\label{thm:construction}
The polynomials
\begin{equation}\label{eq:family}
 A(X)=[Q]_{X^K}[L]_{X^h}\Phi_Q(X^h),\qquad
 B_0(X)=[h]_X[Q]_{X^{hL}}
\end{equation}
are $0$--$1$ masks and satisfy
\[
 A\oplus B_0=\ZZ_M,\qquad |A|=QL,\quad |B_0|=Qh.
\]
Their exact prime-power signatures are
\begin{align*}
 S_A&=\{p^2,q^2\}\cup\{r^{e+1},\ldots,r^{e+m}\},\\
 S_{B_0}&=\{p,q\}\cup\{r,r^2,\ldots,r^e\},
\end{align*}
with the second union empty for $e=0$. Both satisfy (T1), (T2). Every factor
\[
 \Phi_{Q},\Phi_{Qr},\ldots,\Phi_{Qr^e}
\]
is unsupported for $A$.
\end{theorem}
\begin{proof}
The support of $[L]_{X^h}\Phi_Q(X^h)$ is $h\Ap(S;L)$ and lies in $[0,K)$. Its $Q$ translates by multiples of $K$ are disjoint and lie in $[0,M)$. The claim for $B_0$ follows directly from its separated intervals. Telescoping geometric series gives the ordinary identity
\begin{equation}\label{eq:familyproduct}
 A(X)B_0(X)=[M]_X\Phi_Q(X^h).
\end{equation}
For any integer polynomial $R$, $[M]_XR(X)\equiv R(1)[M]_X$ modulo $X^M-1$. Since $\Phi_Q(1)=1$, equation \eqref{eq:tiling} follows.

The complete cyclotomic factorization of $A$ is
\begin{equation}\label{eq:completefactor}
 A(X)=
 \prod_{\substack{s\mid M\\s\nmid K}}\Phi_s(X)
 \prod_{j=e+1}^{e+m}\Phi_{r^j}(X)
 \prod_{j=0}^{e}\Phi_{Qr^j}(X).
\end{equation}
Indeed $\Phi_Q(X^{r^e})=\prod_{j=0}^e\Phi_{Qr^j}(X)$. The first product contains precisely the divisors of $M$ with a $p^2$ or $q^2$ component. This proves the exact signature and all (T2) requirements for $A$. It also excludes the low components of the last $e+1$ factors, proving that they are unsupported. The signature and (T2) of $B_0$ follow either from its factorization or from the complete residue-system argument in Theorem~\ref{thm:complements} below.
\end{proof}

\subsection{The core complement rigidity}
\begin{lemma}[A binary additive matrix]
If a real matrix has the form $F(x,y)=u(x)+v(y)$ and every entry is $0$ or $1$, it depends on at most one variable.
\end{lemma}
\begin{proof}
If $u(x_1)\ne u(x_2)$, the nonzero difference $F(x_1,y)-F(x_2,y)$ is independent of $y$. A nonzero difference between binary entries is $1$ or $-1$; the two rows must therefore be constant rows of opposite values. This forces $v$ to be constant. Otherwise $u$ was already constant.
\end{proof}

\begin{lemma}[All complements of the Ap\'ery core]\label{lem:corecomp}
In $\ZZ_{QL}$, all complements of $C_L=\Ap(S;L)$ are exactly
\[
 a+L\ZZ_{QL},\qquad 0\leq a<L.
\]
\end{lemma}
\begin{proof}
The mask $C_L(X)=[L]_X\Phi_Q(X)$ has, among orders dividing $QL$, exactly the zeros of nontrivial order dividing $L$, and order $Q$. If $C_L\oplus E=\ZZ_{QL}$, the Fourier support of $\ind_E$ is therefore contained in the union of the two coordinate axes under
\[
 \ZZ_{QL}\simeq\ZZ_L\times\ZZ_Q.
\]
Indeed the Fourier transform of the convolution is zero at every nontrivial character, so a nonzero Fourier coefficient of $E$ must occur at a zero of $C_L$. Inverting on the two axes gives $\ind_E(x,y)=u(x)+v(y)$, with $u,v$ chosen real because the original function is real. By the preceding lemma, $\ind_E$ depends on one variable only. Dependence only on $y$ would make $|E|=Q$ a multiple of $L$, impossible since $L>1$ and $\gcd(L,Q)=1$. Thus it selects exactly one value of $x$. Conversely each indicated coset complements the complete residue system $C_L\bmod L$.
\end{proof}

\begin{theorem}[Complete complement parameterization]\label{thm:complements}
For the set $A$ in \eqref{eq:family}, every complement, and no other set, has the form
\begin{equation}\label{eq:allcomplements}
 B=\{t+ha_t+hLj+K\ell_{t,j}:0\leq t<h,\ 0\leq j<Q\},
\end{equation}
where all parameters are independent and
\[
 a_t\in\br L,\qquad \ell_{t,j}\in\br Q.
\]
Consequently the number of complements as subsets of $\ZZ_M$ is
\begin{equation}\label{eq:complementcount}
 L^hQ^{Qh}.
\end{equation}
Every such $B$ satisfies (T1), (T2), with the exact signature $S_B=S_{B_0}$. In the ring $\ZZ[X]$, the monic greatest common divisor of all their masks is
\begin{equation}\label{eq:gcdallB}
 \gcd_{A\oplus B=\ZZ_M}B(X)=[h]_X[Q]_{X^{hL}}=B_0(X).
\end{equation}
\end{theorem}
\begin{proof}
The set $A$ is the full preimage of $hC_L\subset\ZZ_K$ under the quotient by $K\ZZ_M$. In a tiling complement, projection to $\ZZ_K$ is injective: two elements differing by a nonzero member of this period subgroup would give a repeated representation. The projected complement must complement $hC_L$ in $\ZZ_K$.

Split that complement according to its residues $t\bmod h$. In each residue class, division by $h$ gives a complement of $C_L$ in $\ZZ_{QL}$. Lemma~\ref{lem:corecomp} gives the form $a_t+L\br Q$. Each projected point then has an arbitrary lift by $K\ell_{t,j}$. This proves necessity, sufficiency, and uniqueness of the parameters, and hence \eqref{eq:complementcount}.

Modulo $Qh$, the $Q$ entries for each $t$ are all residues with that value modulo $h$, since $L$ is coprime to $Q$. Thus $B$ is a complete residue system modulo $Qh$. It vanishes at every nontrivial $Qh$th root and contains all prime-power factors in $S_{B_0}$. These already contribute $Qh=|B|$ at $X=1$. Any additional prime-power factor would contradict this equality. This proves the exact signature, (T1), and (T2).

For an order $d\mid K$, lifting does not affect evaluation at primitive $d$th roots, and the projected mask is
\[
 [Q]_{X^{hL}}\sum_{t=0}^{h-1}X^{t+ha_t}.
\]
The first factor vanishes whenever $d\mid K$ but $d\nmid hL$. If $d\mid h$, $d>1$, the second factor vanishes. These are exactly the cyclotomic factors of $B_0$. Thus $B_0$ divides every $B$. The choice $a_t=\ell_{t,j}=0$ gives $B=B_0$ itself, so no further factor or higher multiplicity can occur in the gcd.
\end{proof}

\begin{remark}[Quotient periodicity is not actual periodicity]
Formula \eqref{eq:allcomplements} says that the \emph{projection} of $B$ has the indicated fibers. Arbitrary independent lifts usually destroy periods of $B$ in $\ZZ_M$. Universal cyclotomic zeros are inherited, but a subgroup period need not be.
\end{remark}

\begin{corollary}[Three-prime existence at bounded exponents]\label{cor:uniform}
For every $p<q<r$, there is a tiling with unsupported $\Phi_{pqr}$ at period $p^2q^2r^3$, with both factors satisfying (T1) and (T2). The same construction works at period $(pqr)^2$ exactly when $r\in\langle p,q\rangle$. For a fixed admissible $m$, it also gives arbitrarily many unsupported factors while keeping $|A|=pqr^m$ fixed.
\end{corollary}
\begin{proof}
Every integer at least $(p-1)(q-1)$ belongs to $\langle p,q\rangle$, and $r^2>pq>(p-1)(q-1)$. For completeness, if $n\ge(p-1)(q-1)$, choose $b\in\br p$ with $bq\equiv n\pmod p$. Then $n-bq> -p$ and is a multiple of $p$, so it is nonnegative. Thus $n\in\langle p,q\rangle$. Taking $L=r^2$ and $e=1$ proves the first assertion; $L=r$ gives the second. The last assertion follows by increasing $e$ in Theorem~\ref{thm:construction}.
\end{proof}
\section{Sharp small periods and explicit examples}\label{sec:examples}
For $p,q,r=2,3,5$, one has
\[
 C_5=\Ap(\langle2,3\rangle;5)=\{0,2,3,4,6\}.
\]
Modulo $6$, this has the multiset of residues $\{0,0,2,3,4\}$, the union of a $2$-cycle and a $3$-cycle, and modulo $5$ it is a complete residue system. Thus it is a concrete representative of the configuration in \cite[Example 8.1]{KLMS}. The core itself has no unsupported $\Phi_6$, since $2,3\nmid|C_5|$; the following lifts provide the missing cardinality factors.

Taking $m=1,e=0$ in \eqref{eq:family} gives
\[
 M=180,\quad A=\{30j+c:0\le j<6,\ c\in C_5\},\quad B_0=5\br6,
\]
with $S_A=\{4,5,9\}$ and unsupported $\Phi_6$. Taking $m=e=1$ gives
\begin{equation}\label{eq:basic900}
 M=900,\qquad A=\{150j+5c:0\le j<6,\ c\in C_5\}.
\end{equation}
Here $S_A=\{4,9,25\}$ and both $\Phi_6$ and $\Phi_{30}$ are unsupported. The identity
\[
 1+Y^2+Y^3+Y^4+Y^6=\Phi_5(Y)\Phi_6(Y)
\]
gives an immediate direct check of the core. Every complement of \eqref{eq:basic900} is exactly
\begin{equation}\label{eq:900comp}
 \{t+5a_t+25j+150\ell_{t,j}:0\le t<5,\ 0\le j<6\},
 \quad a_t\in\br5,\quad\ell_{t,j}\in\br6.
\end{equation}
There are $5^5 6^{30}$ such complements.

\begin{proposition}[Least cyclic periods]\label{prop:minima}
Among tilings $A\oplus B=\ZZ_M$ with $A$ having an unsupported divisor $\Phi_s$ with $s\mid M$, the least possible $M$ is $180$. If $s$ is required to contain three distinct primes, the least possible $M$ is $900$, and the least possible $|A|$ is $30$.
\end{proposition}
\begin{proof}
An unsupported order contains at least two primes. Each of its primes requires an excluded prime-power level and a different supported level: the first comes from the definition, and the second from (T1), which is necessary for a cyclic tile. By Lemma~\ref{lem:allocation}, both levels are within the specified period. Hence each such prime occurs in $M$ to exponent at least two. If $M$ had at most two distinct prime factors, then $|A|$ would also, so the Coven--Meyerowitz theorem \cite{CM} would give (T2). The two-prime exclusion of \cite[Theorem 1.12]{KLMS} would then rule out an unsupported divisor. Thus $M$ has at least one further prime, and $M\ge2^2 3^2 5=180$. If the unsupported order has at least three primes, the corresponding bound is $M\ge2^2 3^2 5^2=900$, and the cardinality condition gives $|A|\ge30$. The preceding examples attain all these bounds.
\end{proof}

Theorem~\ref{thm:main} does not make the minimal example unique. Appendix~\ref{app:900} classifies all members of $\E(2,3,5)$ by their layers and counts them exactly. In particular, most do not have the period subgroup of the example \eqref{eq:basic900}.
\subsection{An explicit aperiodic minimal example}
In \eqref{eq:layerform}, take the following layers:
\begin{center}
\begin{tabular}{ccl}
\toprule
$w$ & $f_w$ & $g_w$\\
\midrule
$0$ & $(0,0,1)$ & $(0,0)$\\
$1$ & $(0,1,1)$ & $(0,0)$\\
$2$ & $(0,0,0)$ & $(0,1)$\\
$3$ & $(0,0,0)$ & $(0,1)$\\
$4$ & $(0,0,0)$ & $(2,2)$\\
\bottomrule
\end{tabular}
\end{center}
Writing $\zeta=\zeta_3$, their character values sum to $2-2-3\zeta^2-3\zeta^2+6\zeta^2=0$.
The standard CRT inverse is
\[
 (x,y,z)\longmapsto225x+100y+576z\pmod{900}.
\]
It gives
\begin{align*}
 A=\{&0,10,20,60,105,170,180,190,240,255,\\
 &300,310,320,360,375,450,470,490,540,555,\\
 &610,620,630,660,705,750,770,790,825,840\}.
\end{align*}
A complement is the inverse image of $\br2\times\br3\times\br5$:
\begin{align*}
 B_\square=\{&0,1,28,29,100,101,128,153,200,225,\\
 &252,253,325,352,353,425,452,477,504,576,\\
 &577,604,676,677,704,729,776,801,828,829\}.
\end{align*}
Both stabilizers are trivial. For $A$, the nonconstant $g$ in layer $2$ breaks the order-$2$ period; the nonconstant $f$ in layer $0$ breaks the order-$3$ period; the different layers break the order-$5$ period. The standard complement is aperiodic by its product-interval form. They satisfy
\[
 A\oplus B_\square=\ZZ_{900},\qquad
 S_A=\{4,9,25\},\quad S_{B_\square}=\{2,3,5\},\quad\Phi_{30}\mid A.
\]
Since affine automorphisms preserve the stabilizer size, this set is not affinely equivalent to the periodic example \eqref{eq:basic900}. The aperiodic example is still contained in $5\ZZ_{900}$ and is not primitive.

\subsection{The lower modulus 180 also allows both factors aperiodic}\label{sec:180aperiodic}
Divide the preceding $A$ by $5$ and take the zero-modulo-$5$ slice of $B_\square$, also divided by $5$. This gives
\begin{align*}
 C=\{&0,2,4,12,21,34,36,38,48,51,60,62,64,72,75,\\
     &90,94,98,108,111,122,124,126,132,141,150,154,158,165,168\},\\
 D=\{&0,20,40,45,65,85\}.
\end{align*}
They tile $\ZZ_{180}$, with $S_C=\{4,5,9\}$ and $S_D=\{2,3\}$, and $\Phi_6$ is unsupported for $C$. Any period of $C$ would lift to a period of $A$, so $C$ is aperiodic. In CRT coordinates $\ZZ_4\times\ZZ_9\times\ZZ_5$, the complement is
\[
 D=\{0,1\}\times\{0,2,4\}\times\{0\},
\]
which is aperiodic as well. Moreover $C$ is primitive, since it contains $0,2,21$. Thus the least modulus $180$ for any unsupported order can be attained with both factors aperiodic and with the unsupported factor's set primitive. For a three-prime unsupported order itself, the least modulus $900$ can be attained with both factors aperiodic. Every $A$ with the specific unsupported order $30$ at this modulus is necessarily nonprimitive.
\section*{Acknowledgments}

The authors thank the authors of the works cited in this paper for the
questions and ideas that motivated the present study. OpenAI's GPT Astra
was used during the preparation of the manuscript for language editing
and as a research aid in exploring and identifying useful examples.
All mathematical statements, constructions, and proofs were independently
verified by the authors, who take full responsibility for the contents
of the paper.

The computational checks, including those implemented in several Python files mentioned in Appendix \ref{app:900}, supplement the proofs and can be reproduced using the verification code available in the \href{https://github.com/TanHu1999/Unsupported-Cyclotomic-Divisors-in-Three-Prime-Integer-Tilings}{\textcolor{linkblue}{public GitHub repository}}.

\appendix
\section{Complete enumeration at modulus 900}\label{app:900}
For $p,q,r=2,3,5$, Corollary~\ref{cor:layerform} is a finite complete classification. For each of the five high-$5$ layers, choose
\[
 f\in\{0,1\}^3,\qquad g\in\{0,1,2\}^2,
 \qquad f\text{ constant or }g\text{ constant}.
\]
There are $2\cdot3^2+3\cdot2^3-2\cdot3=36$ distinct layer choices. Their only remaining constraint is the exact zero sum \eqref{eq:layerzero}. This is a necessary-and-sufficient description of \emph{all} the sets in question; it does not presuppose periodicity or a particular polynomial factorization.

\subsection{The constant-term certificate}
Put $\zeta=\zeta_3$. A layer character value is one of
\[
 \pm6\zeta^j\ (\text{multiplicity }1),\quad
 \pm3\zeta^j\ (\text{multiplicity }2),\quad
 \pm2\zeta^j\ (\text{multiplicity }3),\qquad j=0,1,2.
\]
For example, if both low vectors are constant the value is $\pm6\zeta^j$. If only $f$ is constant, two distinct $g$ values sum to the negative of the third root. If only $g$ is constant, the three signs have sum $\pm1$.

In the integer basis $(1,\zeta)$, define
\[
 K_s(U,V)=U^s+U^{-s}+V^s+V^{-s}+(UV)^s+(UV)^{-s},\qquad
 F=K_6+2K_3+3K_2.
\]
Then the number of valid ordered five-layer choices for a fixed low-$5$ residue is
\begin{equation}\label{eq:cttotal}
 \CT F^5=366\,960.
\end{equation}
A period of order $2$ is equivalent to every $g_w$ being constant; a period of order $3$ is equivalent to every $f_w$ being constant. Thus the other counts are
\begin{align}
 \CT(K_6+3K_2)^5&=146\,160,\label{eq:ct2}\\
 \CT(K_6+2K_3)^5&=46\,200,\label{eq:ct3}\\
 \CT K_6^5&=360.\label{eq:ct6}
\end{align}
No order-$5$ period is possible: it would make all layers identical, while each single-layer character value is nonzero. A stabilizer order divides $|A|=30$, so no other orders need be considered. Inclusion--exclusion, followed by multiplication by the five choices of the common low-$5$ digit, gives:
\begin{center}
\begin{tabular}{rrr}
\toprule
Stabilizer order & One fixed residue modulo $5$ & All five residues\\
\midrule
$1$ & $174\,960$ & $874\,800$\\
$2$ & $145\,800$ & $729\,000$\\
$3$ & $45\,840$ & $229\,200$\\
$6$ & $360$ & $1\,800$\\
\midrule
Total & $366\,960$ & $1\,834\,800$\\
\bottomrule
\end{tabular}
\end{center}
These are counts of distinct subsets of $\ZZ_{900}$, not counts modulo affine equivalence.

The integer constants in \eqref{eq:cttotal}--\eqref{eq:ct6} are checked by two independent finite expansions in the repository's \texttt{verify.py} and in the supplementary \texttt{verification.py}: a coefficient-convolution recurrence and an unordered five-term enumeration weighted by multinomial coefficients. Equivalently, for the coefficient function $c$ of any indicated Laurent polynomial, set $D_0(0,0)=1$ and $D_0(x,y)=0$ for $(x,y)\ne(0,0)$, and use
\[
 D_{k+1}(x,y)=\sum_{(a,b)}c(a,b)D_k(x-a,y-b).
\]
The desired constant is $D_5(0,0)$. The self-contained supplementary file \texttt{verification.py} uses only the Python standard library and is run with \texttt{python3 verification.py}. Besides the two constant-term computations, it independently counts actual ordered layers, enumerates all complements of each of the $360$ cores below, and checks the explicit tiling examples by exact integer arithmetic. The structural classification and absence of other stabilizer types are proved above, not inferred from enumeration.

\subsection{The periodic core subclass}
For completeness, the order-$6$-periodic subfamily consists exactly of the sets
\[
 A_{a,C}=a+\{150j+5c:0\le j<6,\ c\in C\},\qquad 0\le a<5,
\]
where $C\subset\ZZ_{30}$ is a five-point complete residue system modulo $5$ satisfying $\Phi_6\mid C$. Indeed, quotienting by the order-$6$ period subgroup, then subtracting the common residue modulo $5$ and dividing by $5$, gives precisely such a core; the converse follows from the layer form. There are $360$ such cores.

To see the count directly, let $c_j$ be the six residue counts modulo $6$. Since $\zeta_6=-\zeta_3^2$, the zero condition implies that the difference between the even and odd counts in each of the three $\zeta_3$ phases is a common integer $k$. The total is five, so $k=1$ or $-1$. In the majority parity, the three counts are $(2,1,1)$; in the minority parity they are $(1,0,0)$, with its nonzero position opposite the doubled residue. There are six such residue patterns. Assigning the five distinct residue labels modulo $5$ to a pattern gives $5!/2=60$ cores. Hence the count is $6\cdot60=360$.

Under affine automorphisms $x\mapsto ux+v$ of $\ZZ_{30}$, with $\gcd(u,30)=1$, these cores have two orbits, of sizes $120$ and $240$, represented by
\[
 \{0,2,3,4,6\},\qquad \{0,1,3,9,17\}.
\]
One description of the distinction is whether the unique minority-parity point, viewed modulo $5$, is the midpoint of the two points sharing the doubled residue modulo $6$. Affine transformations preserve this property. After normalizing the doubled residue modulo $6$ and its two labels modulo $5$, the minority label is either the midpoint or one of the two remaining labels; the available reflections identify the latter two possibilities. The remaining two singletons may be exchanged independently in the modulo-$6$ coordinate. This proves there are exactly two orbits. Their sizes also follow from the ratio $1:2$ of midpoint and nonmidpoint choices, and are independently checked in the program.

Both orbit representatives, and hence all $360$ cores, have among nontrivial divisor orders of $30$ exactly the zeros $5$ and $6$. To exclude orders $10,15,30$ without factoring, fix $d\in\{2,3,6\}$ and index the points of a core as $c_t\equiv t\pmod5$. Vanishing at order $5d$ is equivalent to
\[
 \sum_{t=0}^4\zeta_d^{c_t}\zeta_5^t=0,
\]
since $x\mapsto\zeta_d^x\zeta_5^x$ has order $5d$. The minimal polynomial of $\zeta_5$ over $\QQ(\zeta_d)$ is $1+X+X^2+X^3+X^4$: the field extension has degree $4$ because $\gcd(d,5)=1$. Therefore the five coefficients $\zeta_d^{c_t}$ would all have to be equal, forcing the points to lie in one residue class modulo $d$. The five-point residue pattern modulo $6$ established above rules this out for each $d$. Orders $2,3$ are excluded by the cardinality $5$. The Fourier proof of Lemma~\ref{lem:corecomp} therefore applies to every such core. Consequently all $1\,800$ order-$6$-periodic members of $\E(2,3,5)$ have exactly the same complement family \eqref{eq:900comp}: for fixed $a$, slicing a complement by its residues modulo $5$ again gives the same coset complements of $C$.

\end{document}